\documentclass[11pt,reqno]{amsart}

\usepackage{amsmath,amssymb,amsthm,amsfonts}
\usepackage{graphicx}
\usepackage{enumerate}

\usepackage[paperwidth=8.5in,paperheight=11in,
left=1.125in, right=1.125in,top=1.25in, bottom=1.25in]{geometry} 

\newtheorem{theorem}{Theorem}[section]
\theoremstyle{plain}

\newtheorem{corollary}[theorem]{Corollary}
\newtheorem{definition}[theorem]{Definition}
\newtheorem{example}{Example}
\newtheorem{lemma}[theorem]{Lemma}
\newtheorem{proposition}[theorem]{Proposition}
\newtheorem{remark}{Remark}
\numberwithin{equation}{section}

\newcommand{\CC}{\mathbb{C}}

\newcommand{\dbb}{\bar{\partial}_b}

\DeclareMathOperator{\rk}{rank}

\def \d{\partial}

\def\dbar{\overline{\partial}}
\def \hbar{\overline{h}}

\def\\GBB{\cal B}
\def\alphabar{\overline{\alpha}}
\def\betabar{\overline{\beta}}

\def\gammabar{\overline{gammabar}}

\def\deltabar{\overline{\delta}}
\def\1bar{\overline{1}}
\def\2bar{\overline{2}}

\def\lbar{\overline{\ell}}

\def\jbar{\bar{j}}
\def\kbar{\overline{k}}

\def\lbar{\overline{\ell}}

\def\hbar{\overline{h}}
\def\d{\partial}
\def\gammabar{\overline{\gamma}}
\def\deltabar{\overline{\delta}}

\def\ibar{\bar{i}}
\allowdisplaybreaks

\begin{document}
\title[CR-Analogue of Siu-$\d\dbar$-formula and Applications]{CR-Analogue of Siu-$\d\dbar$-formula and Applications to Rigidity problem for pseudo-Hermitian harmonic maps}
\begin{abstract}
 We give several versions of Siu's $\d \dbar$-formula for maps from a strictly
 pseudoconvex pseudo-Hermitian manifold $(M^{2m+1}, \theta)$ into a K\"{a}hler manifold $(N^n, g)$. We also define and study the notion of pseudo-Hermitian harmonicity for maps from $M$ into $N$. In particular, we prove a CR version of Siu Rigidity Theorem for pseudo-Hermitian
harmonic maps from a pseudo-Hermitian manifold with vanishing Webster torsion into a K\"{a}hler manifold having strongly negative curvature.
\end{abstract}
\author{Song-Ying Li}
\address{Department of Mathematics, University of California, Irvine, CA 92697}
\email{sli@math.uci.edu}
\author{Duong Ngoc Son}
\address{Department of Mathematics, University of California, Irvine, CA 92697}
\email{snduong@math.uci.edu}
\date{July 2, 2019}
\maketitle

\section{Introduction}
Let $(M, h)$ and $(N,g)$ be Riemannian manifolds and $f \colon M\to N$. Then $f$ is said to be harmonic if 
\begin{equation}
\Delta_M f^{\alpha} + \Gamma^{\alpha}_{\beta\gamma}(f) \frac{\partial f^{\beta}}{\partial y^i} \frac{\partial f^{\gamma}}{\partial y^j} h^{ij} = 0,
\end{equation}
where $y^i$ are the coordinates on $M$, $\Gamma^{\alpha}_{\beta\gamma}$ are the Christoffel symbols of $N$, and $\Delta_M$ is the Laplace-Beltrami operator of $M$. In \cite{Siu}, Siu proved the following theorem for harmonic maps between K\"{a}hler manifolds which implies his celebrated strong rigidity theorem.
\begin{theorem}[Siu Rigidity Theorem]
Suppose that $f\colon M \to N$ is a harmonic map between K\"{a}hler manifolds. If $M$ is a closed manifold, $N$ has strongly negative curvature in the sense of Siu, and $df$ has real rank at least 4 at some point $p\in M$, then $f$ is either holomorphic or anti holomorphic.
\end{theorem}
This theorem, together with existence theorem for harmonic maps by Eells and Sampson \cite{ES}, 
gives the following strong rigidity result for compact K\"{a}hler manifolds of strongly negative curvature: suppose that $M$ 
and $N$ are two K\"{a}hler manifolds of complex dimension at least $2$ with $M$ being closed and suppose that $N$ has strongly negative curvature. Then
$M$ and $N$ are topologically equivalent if and only if they are biholomorphically equivalent.

For strictly pseudoconvex pseudo-Hermitian CR manifolds, beside Laplace-Beltrami operator associated with 
the Webster metric, there are other notions of Laplacian (e.g, the sub-Laplacian and Kohn-Laplacian), which lead to several notions of harmonicity for maps from CR manifolds. For instance, the pseudoharmonic maps defined by using sub-Laplacian have been studied by many authors (see, e.g, \cite{BDU, DK, Pet}). (However,
our notion of pseudo-Hermitian harmonic maps defined below does \emph{not} coincide with the notion of pseudoharmonic maps into Riemannian manifolds as defined in~\cite{DK}). Motivated by these research and our own work with X. Wang \cite{LSW} on Kohn-Laplacian, we define the notion of \emph{pseudo-Hermitian} harmonic maps, using Kohn-Laplacian on CR manifolds, 
and study the rigidity analogous to Siu's strong rigidity.

Let $(M^{2m+1}, \theta)$ be a strictly pseudoconvex pseudo-Hermitian manifold with a pseudo-Hermitian structure $\theta$ and the Levi metric $h$, and let $(N^n, g)$ be a K\"ahler manifold
with K\"abler metric $g$. For any differentiable map $f: M\to N$, we define the \emph{total $\dbar_b$-energy functional} of $f$ by 
\begin{equation}\label{e:dbar-energy} 
E[f] = \int_M g_{\alpha\betabar} f^{\alpha}_{\ibar} f^{\betabar}_{j} h^{j\ibar} \theta \wedge (d\theta)^m.
\end{equation}
We say that $f$ is \emph{pseudo-Hermitian harmonic} if $f$ is a critical point of the functional $E[\cdot]$. Write $e_{\alpha} = \partial /\partial z^{\alpha}$, where $z^{\alpha}$ is a local holomorphic coordinate on $N$, and let
\begin{equation}
\Gamma^{\alpha}_{\beta\gamma} = \frac{\partial g_{\gamma\deltabar}}{\partial z^{\beta}} g^{\deltabar \alpha}
\end{equation}
be the corresponding Christoffel symbols.
Then, the Euler-Lagrange equation for 
$E[\cdot]$ is 
\begin{equation}
\tau[f] = h^{j\bar{i}} \left( f^{\alpha}_{\ibar j} 
 +\Gamma^{\alpha}_{\beta\gamma} f^{\beta}_{\bar{i}} f^{\gamma}_{j}\right) e_{\alpha}=0.
\end{equation}
The key ingredient in Siu's proof is his celebrated $\d\dbar$-formula, which does not involve Ricci
curvature of the source manifold $M$. In order to 
prove a Siu-type theorem in CR geometry, the important step is to find an analogue of the
$\d\dbar$-formula. However, since the Tanaka-Webster connection always has torsion, our formula should involve torsion of $M$. To be more precise, let $f^{\alpha}_{\jbar | k | l}$ be the components
of $DD\bar{\partial}_b f$, where $D$ is the connection induced by Tanaka-Webster connection on $M$ and 
the pull-back of the complexified Levi-Civita connection on $N$ (see Section~3 for detail). Following Graham and Lee \cite{GL}, we 
define the second and third order operators
\begin{align}
&f^\alpha_{i|\jbar}=f^\alpha_{i\jbar}+{}^N\Gamma^\alpha_{\sigma\rho} f^\sigma_i f^\rho_{\jbar};\\
& B_{i\jbar} f^{\alpha} = f^{\alpha}_{i | \jbar} - \frac{1}{m} (f^{\alpha}{}_{k | \bar{l}} h^{k \bar{l}})h_{i\jbar};\\
& P_i f^{\alpha} = f^{\alpha}_{\jbar | l | i} h^{l\jbar} + m\sqrt{-1}
A_{i}\!^{\bar{j}} f^{\alpha}_{\jbar};\\
&P f = \left( P_i f^{\alpha} \right) \theta^i \otimes e_{\alpha}.
\end{align}
Thus, $Pf$ is a $f^{\ast} T^{1,0}N$-valued (1,0)-form on $M$. We can contract $Pf$ with $\bar{\partial}_b \bar{f}$ to
obtain a scalar, namely
\begin{equation}
\langle Pf , \bar{\partial}_b \bar{f} \rangle = g_{\alpha\bar{\beta}} (P_i f^{\alpha}) f^{\betabar}_{\jbar} h^{i\jbar}
\end{equation}
We also define the norm of the tensor $B_{i\jbar}f$ by
\begin{equation}
 |B_{i\jbar} f^{\alpha}|^2 = g_{\alpha\betabar}(B_{i\jbar} f^{\alpha})(\overline{B_{k\bar{l}} f^{\beta})} h^{i\bar{k}} h^{l\jbar}. 
\end{equation}
We say that $f$ is a \emph{CR-pluriharmonic} map if $B_{\ibar j}(f^\alpha)=0$ for all $1\le i, j\le m$ and $1\le \alpha\le n$. Now we can state our main theorem.

\begin{theorem}\label{thm:crsiu}
Let $(M^{2m+1}, \theta)$ be a closed pseudo-Hermitian CR manifold and let $(N^n, g)$ be a K\"ahler manifold. Let $f:M\to N$ be a smooth map. Then
\begin{multline}\label{e:po}
 -\frac{m-1}{m} \int_M \langle Pf, \bar{\partial}_b \bar{f}\rangle 
 = \int_M |B_{i\jbar}f^{\alpha}|^2 
 + \int_M R_{\rho\deltabar\gamma\betabar} f^{\betabar}_{\bar{l}} f^{\rho}_{\jbar}
(f^{\gamma}_{i} f^{\deltabar}_{k} - f^{\gamma}_{k} f^{\deltabar}_{i}) h^{i\bar{l}} h^{k\jbar},
\end{multline}
and
\begin{equation}\label{e:po2}
-\int_M \langle Pf, \bar{\partial}_b \bar{f}\rangle 
=\int_M \langle \tau[f], \overline{\bar{\tau}[f]} \rangle 
-\sqrt{-1}m\int_M g_{\alpha\betabar}A^{\bar{i}\bar{j}} f^\alpha_{\ibar}
f^{\betabar}_{\jbar}. 
\end{equation}
\end{theorem}
Here, the curvature tensor $R_{\alpha\betabar\gamma\deltabar}$ is given by
\begin{equation}
 R_{\alpha\betabar\gamma\deltabar} 
 = \frac{\partial^2 g_{\gamma\deltabar}}{\partial z^{\alpha} \partial z^{\betabar}}
 - g^{\mu\bar{\nu}} \frac{\partial g_{\gamma\bar{\nu}}}{\partial z^{\alpha}} 
 \frac{\partial g_{\mu\deltabar}}{\partial z^{\betabar}}.
\end{equation}

In \cite{Siu}, the author introduced the following definition: the curvature tensor $R_{\alpha\betabar\gamma\deltabar}$ is said to be strongly negative (resp. strongly seminegative) if
\(
 R_{\alpha\betabar\gamma\deltabar} (A^{\alpha}\overline{B^\beta} - C^{\alpha}\overline{D^\beta})
 (\overline{A^{\delta}\overline{B^\gamma} - C^{\delta}\overline{D^\gamma}})
\)
is positive (resp. nonnegative) for arbitrary complex numbers $A^\alpha$, $B^\alpha$, $C^\alpha$, $D^\alpha$, when 
$A^{\alpha}\overline{B^\beta} - C^{\alpha}\overline{D^\beta} \ne 0$ for at least one pair of indices $\alpha, \beta$. Using argument in \cite{Siu}, we obtain the following corollary.
\begin{corollary}\label{cor:rigidity}
Let $(M^{2m+1}, \theta)$ be a closed pseudo-Hermitian CR manifold of dimension at least $5$, and $(N^n, g)$ a K\"{a}hler manifold. Then
\begin{enumerate}[(a)]
\item If $N$ has strongly semi-negative curvature, then
\begin{equation}\label{e:positivity}
 \int_M \langle Pf, \bar{\partial}_b \bar{f}\rangle 
 =:\int_M g_{\alpha \betabar} (P_if^{\alpha}) f^{\betabar}_{\jbar} h^{i\jbar}\ \theta \wedge ( d\theta)^m \leq 0
\end{equation}
for any smooth map $f \colon M \to N$. The equality holds if and only if $f$ is CR-pluriharmonic.
\item If $M$ has vanishing pseudohermitian torsion, $N$ has strongly negative curvature, and if $f$ is pseudo-Hermitian harmonic with $df$ having rank at least $4$ at a dense set
of $M$, then $f$ must be CR-holomorphic or anti CR-holomorphic.
\end{enumerate}
\end{corollary}
We remark that part (a) in the corollary generalizes a result in \cite{GL} about positivity of the operator $P$ for maps; and part (b) gives a CR version 
of Siu's rigidity theorem.

\section{Harmonic map equations} \label{sec:main}

For basic notions in pseudohermitian geometry, we refer the reader to \cite{L, L2} or \cite{We}, \cite{LiLuk}, \cite{LiW} and \cite{LSW}.
Let $(M^{2m+1},\theta)$ be a $(2m+1)$-dimensional, strictly pseudoconvex pseudo-Hermitian manifold and 
let $(N^n, g)$ be a
K\"ahler manifold. Suppose that $f\colon M \to N$ is a smooth map.
We can define the \emph{pointwise $\bar{\partial}_b$-energy} of $f$ as follows.
Suppose $p\in M$ and $q=f(p)$. We choose a local coordinate chart $V$
of $N$ near $q$. Near $p$, we choose a local holomorphic frame $\{Z_{i}\}$
and let $h_{i\bar{j}}$ be the Levi-form with respect to $\{Z_{i}\}$. That is,
\begin{equation}
d\theta = ih_{i\jbar} \theta^i \wedge \theta^{\jbar},
\end{equation}
where, $\{\theta^i\}$ is a holomorphic coframe dual to $\{Z_i\}$ and $\theta^{\ibar} = \overline{\theta^{i}}$.
Then
the $\bar{\partial}_b$-energy density $e(f)(p)$ is
\begin{equation}
 e(f)(p) = g_{\alpha\bar{\beta}} f^{\alpha}_{\bar{i}} f^{\bar{\beta}}_{j} h^{j\bar{i}},
\end{equation}
where the summation convention is used. The $\bar{\partial}_b$-energy functional of $f$
is
\begin{equation}
 E[f] = \int_M e(f)(p) =\int_M e(f) \theta\wedge (d\theta)^m.
\end{equation}
 A critical point of the functional $E$ 
satisfies $\tau(f) = 0$, where
\begin{equation}\label{e:ph}
\tau^\alpha[f]=: h^{j\bar{i}} \left( f^{\alpha}_{\ibar j} 
 +\Gamma^{\alpha}_{\beta\gamma} f^{\beta}_{\bar{i}} f^{\gamma}_{j}\right),\quad 1\le \alpha\le n.
\end{equation}
Here, $\Gamma^{\alpha}_{\beta\gamma}$, or more precisely, $\Gamma^{\alpha}_{\beta\gamma} \circ f$, is the Christoffel symbols of $N$ evaluated
at the point $f(p)$, with respect
to coordinates $\{z^{\alpha}\}$ and $f^{\alpha}_{\bar{i}j}$ is the second order 
covariant derivative of $f^{\alpha}$ with respect to the Tanaka-Webster connection 
on $M$. In fact, given $p\in M$, we choose a local chart $V$ of $f(p)$ and consider 
the family $f_t$ defined as $f_t^{\alpha} = f^{\alpha} + t\psi^\alpha$, where $\psi^{\alpha}$ 
are smooth functions with compact supports in a neighborhood of $p$. Then, by a standard calculation 
for ${d \over d t} E(f_t)|_{t=0}=0$, one has Euler-Lagrange equation \eqref{e:ph}.

\begin{definition} Let $(M^{2m+1},\theta)$ be a $(2m+1)$-dimensional pseudohermitian manifold and let $(N, g)$ is a K\"{a}hler 
manifold with K\"ahler metric $g$. Suppose that $f \colon M \to N$ is a smooth map. 
\begin{enumerate}[(i)]
 \item We say that $f$ is pseudohermitian 
 harmonic map if it satisfies $\tau[f] = 0$. That is,
\begin{equation}
 h^{j\bar{i}} \left( f^{\alpha}_{\bar{i}j} 
 +\Gamma^{\alpha}_{\beta\gamma} f^{\beta}_{\bar{i}} f^{\gamma}_{j}\right) e_{\alpha} = 0
\end{equation}
 \item We say that $f$ is $\bar{\partial}_b$-pluriharmonic map if
 \begin{equation}\label{e:dbph} 
 \left(f^{\alpha}_{\bar{i}j} 
 +\Gamma^{\alpha}_{\beta\gamma} f^{\beta}_{\bar{i}} f^{\gamma}_{j}\right) \theta^{\ibar} \wedge \theta^{j} \otimes e_{\alpha} = 0
\end{equation}
 \item We say that $f$ is CR-pluriharmonic if $m\geq 2$ and 
 \begin{equation}\label{e:crph}
 f^{\alpha}_{\bar{i}j} 
 +\Gamma^{\alpha}_{\beta\gamma} f^{\beta}_{\bar{i}} f^{\gamma}_{j} = {\tau^{\alpha}[f] \over m} h_{j\ibar}.
 \end{equation}
\end{enumerate}
\end{definition}
\begin{remark}\rm
\begin{enumerate}[(a)] 
 \item It is easy to see that $f$ is $\bar{\partial}_b$-pluriharmonic map if and only if
 $f$ is both $CR$-pluriharmonic and pseudohermitian harmonic.
 \item If $f$ is a CR map,
then regardless of the K\"{a}hler metric on $N$ and pseudohermitian structure on $M$, $f$ is $\bar{\partial}_b$-pluriharmonic. On the other hand, a conjugate 
(anti) CR map is $\bar{\partial}_b$-pluriharmonic if and only if $df(T)=0$, where $T$ is the Reeb vector field associated to the contact form~$\theta$.
 \item When $N = \mathbb{C}^n$ and $m\geq 2$, from a well-known result by Bedford, Ferderbush \cite{B,BF} and Lee \cite{L2}, $f$ is CR-pluriharmonic if and only if for any holomorphic local coordinates $\{z^\alpha\}$ on $N$, the real and imaginary parts of $f^{\alpha}: = z^\alpha\circ f$ are locally real parts of CR functions.
\end{enumerate}
\end{remark}
\begin{example}\rm Let $N=\mathbb{B}_n$ be the unit ball in $\mathbb{C}^n$ with Bergman metric given in standard coordinates by
\begin{equation}
 g_{\alpha\betabar} = (1-|z|^2)^{-2} \left(\overline{z^{\alpha}} z^{\beta} + (1-|z|^2) \delta_{\alpha\beta} \right).
\end{equation}
The corresponding Christoffel symbols are
\begin{equation}
 \Gamma^{\alpha}_{\beta\gamma} 
 = (1-|z|^2)^{-1} (\overline{z^{\beta}} \delta_{\alpha\gamma} + \overline{z^{\gamma}} \delta_{\alpha\beta}).
\end{equation}
Therefore, a map $f\colon M \to \mathbb{B}_n$ is pseudohermitian harmonic if and only if $(f^{\alpha})$ satisfies the following system.
\begin{equation}\label{e:crharbergman}
 -\Box_b f^{\alpha} + (1-|z|^2)^{-1} h^{i\jbar} 
 \sum_{\beta} \left( f^{\alpha}_{i} f^{\betabar} f^{\beta}_{\jbar} + f^{\alpha}_{\jbar} f^{\betabar}f^{\beta}_{i} \right)=0,
\end{equation}
where $\Box_b$ is the Kohn-Laplacian. It is well-known that the Bergman metric has strongly negative curvature \cite{Siu}.
Corollary~\ref{cor:rigidity} implies that any smooth embedding of the sphere $S^{2m+1}$ ($m\geq 2$) into $\mathbb{B}_n$ satisfying
\eqref{e:crharbergman} must be a CR embedding.
\end{example}

The following proposition shows that the CR-pluriharmonicity is CR invariant (i.e does not depend on the pseudo-Hermitian structures on~$M$). 
When $N = \mathbb{C}^n$, the proposition follows directly from aforementioned result in \cite{L2}.
\begin{proposition}\label{prop:crinvar}
Let $(M,\theta)$ and $(N, g)$ be a pseudohermitian manifold and a K\"{a}hler manifold, respectively. If $f \colon M \to N$ is CR-pluriharmonic 
with respect to $\theta$, then it is CR-pluriharmonic with respect to any $\hat{\theta}=e^{2\sigma} \theta$.
\end{proposition}
\begin{proof} Locally, we can choose a local holomorphic frame $\{Z_{i}\}$ and its dual admissible coframe $\{\theta^{i}\}$ for $\theta$. As in \cite{L}, we choose
\begin{equation}
\hat{\theta}^{k} = e^{\sigma}(\theta^{k} + 2\sqrt{-1} \sigma^{k} \theta).
\end{equation}
Then $\{\hat{\theta}^{k}\}$ is an admissible coframe for $\hat{\theta}$, with the same matrix $\hat{h}_{i\jbar} = h_{i\jbar}$, and dual to the holomorphic frame $\{\hat{Z}_{k} = e^{-\sigma} Z_k\}$.
The Webster connection forms $\hat{\omega}_{i}\!^{l}$ is given by \cite{L}
\begin{align}
\hat{\omega}_{i}\!^{l}
= \omega_{i}\!^{l} + 2(\sigma_{i} \theta^{l} - \sigma^{l}\theta_{i}) 
+ \delta_{i}\!^{l} (\sigma_k \theta^k - \sigma^{k}\theta_{k})\notag \\
 + \sqrt{-1}(\sigma^{l}\!_{i} + \sigma_{i}\!^{l} + 4\sigma_{i}\sigma^{l} + 4\delta_{i}\!^{l} \sigma_k \sigma^k) \theta.
\end{align}
Therefore,
\begin{align}
\hat{\omega}_{i}\!^{l} (\hat{Z}_{\bar{j}}) 
= e^{-\sigma}({\omega}_{i}\!^{l} ({Z}_{\bar{j}}) -2\sigma^{l}h_{i\jbar} - \delta_i\!^{l} \sigma_{\jbar}).
\end{align}
In local frame $\hat{Z}_k$, we write $\hat{\nabla}_{i} = \hat{\nabla}_{Z_{i}}$, etc.,
\begin{align}
\hat{\nabla}_{\jbar}\hat{\nabla}_i f^{\alphabar}
 &= \hat{Z}_{\jbar}(\hat{Z}_{i} f^{\alphabar}) \notag
 - \hat{\omega}_{i}\!^{l}(\hat{Z}_{\jbar})\hat{Z}_{l} f^{\alphabar} \notag\\
& = e^{-2\sigma}\bigl(Z_{\jbar}(Z_i f^{\alphabar}) -\sigma_{\jbar} Z_i f^{\alphabar} \notag \\ 
& \quad - \omega_i\!^l (Z_{\jbar}) Z_{l} f^{\alphabar} + 2\sigma^{l}h_{i\jbar}Z_{l} f^{\alphabar} + \delta_i\!^{l} \sigma_{\jbar}Z_{l} f^{\alphabar}\bigr) \notag\\
 &= e^{-2\sigma} \nabla_{\jbar} \nabla_{i} f^{\alphabar} + 2 e^{-2\sigma} \sigma^l f^{\alphabar}_{l} h_{i\jbar}.
\end{align}
Suppose that $f$ is CR-pluriharmonic, i.e. \eqref{e:crph} holds, then we obtain
\begin{equation}
\hat{\nabla}_{j}\hat{\nabla}_{\ibar} f^{\alpha} + \Gamma^{\alpha}_{\beta\gamma} (\hat{Z}_{i} f^{\beta})(\hat{Z}_{j} f^{\gamma})
= e^{-2\sigma} \left(\frac{\tau^{\alpha}[f]}{m} +2\sigma^{\bar{l}} f^{\alpha}_{\bar{l}} \right) h_{j\ibar},
\end{equation}
which implies that $f$ is CR-pluriharmonic with respect to $\hat{\theta}$, as desired.
\end{proof}

\section{Proof of Theorem 1.2}

In this section, we give a proof of Theorem~\ref{thm:crsiu}. For convenience, we introduce several notations similar to those in \cite{Sam1}. Let $\bar{\partial}_b f$ be the $f^{\ast} T^{0,1} N$-valued one form represented by $(f^{\alpha}_{\ibar})$. Then the covariant
derivative $D\bar{\partial}_b f$ in \cite{Siu} has components
\begin{equation}\label{e:covariantderivative}
f^{\alpha}_{\ibar | j} 
=: f^{\alpha}_{\ibar j} + \Gamma^{\alpha}_{\beta\gamma} f^{\beta}_{\ibar} f^{\gamma}_{j},
\end{equation}
where, $f^{\alpha}_{\ibar j}$ denotes covariant derivative with respect to Tanaka-Webster 
connection on $M$. Thus, $f$ is $\bar{\partial}_b$-pluriharmonic if and only if $D\bar{\partial}_b f = 0$, 
and $f$ is pseudohermitian
harmonic if and only if
\begin{equation}
\tau^\alpha[f]=: h^{j\ibar} f^{\alpha}_{\ibar | j} = 0.
\end{equation}
The covariant derivative $\bar{D}\bar{D} \partial_b f$ is formed similarly, namely
\begin{align}
	f^{\alpha}_{i|\jbar | \bar{k}} 
	& = Z_{\bar{k}}(f^{\alpha}_{i|\jbar}) - \Gamma_{\bar{k} \jbar}^{\bar{l}} f^{\alpha}_{i | \bar{l}}
	- \Gamma^l_{\kbar i} f^{\alpha}_{l|\jbar}
	+ \Gamma^{\alpha}_{\beta\gamma} f^{\beta}_{i| \bar{j}} f^{\gamma}_{\bar{k}} = (f^{\alpha}_{i | \jbar})_{,k} + \Gamma^{\alpha}_{\beta\gamma} f^{\beta}_{i| \bar{j}} f^{\gamma}_{\bar{k}}.
\end{align}
Here, the covariant derivatives with respect to Tanaka-Webster connection on $M$ are denoted by indexes preceded by commas. Also,
\begin{align}
 f^{\alpha}_{i|\jbar | k} 
 = (f^{\alpha}_{i|\jbar})_{,k} + \Gamma^{\alpha}_{\beta\gamma} f^{\beta}_{i| \jbar} f^{\gamma}_{k}\quad\hbox{and }\ 
 f^{\alpha}_{i | j | \kbar} 
 = (f^{\alpha}_{i | j})_{,\kbar} + \Gamma^{\alpha}_{\beta\gamma} f^{\beta}_{i| j} f^{\gamma}_{\kbar}.
\end{align}
In what follows, we will denote the curvatures on $M$ and $N$ by $R$ with Latin indices and Greek indices, respectively. Thus, $ R_{i}{}^{l}{}_{j\bar{k}}$ and $A_{ij}$ are components of Webster curvature and torsion on $M$, respectively, while $R^{\alpha}_{\beta\gammabar \delta}$ are components of curvature on~$N$.
\begin{lemma} We have the following commutation relations
\begin{align}
 & f^{\alpha}_{i | \jbar} - f^{\alpha}_{\jbar | i} 
 = \sqrt{-1} f^{\alpha}_0 h_{i\jbar},\quad f^\alpha_{i|j}=f^\alpha_{j|i}; \label{e:c1}\\
 & f^{\alpha}_{0| \ibar} - f^{\alpha}_{\ibar | 0} = A_{\ibar}{}^{k} f^{\alpha}_{k}; \label{e:c2}\\
 & f^{\alpha}_{i|\jbar | \bar{k}} - f^{\alpha}_{i| \bar{k} | \jbar}
 = \sqrt{-1}\left(h_{i\jbar} A_{\bar{k}}{}^{l} - h_{i\bar{k}} A_{\jbar}{}^{l}\right)f^{\alpha}_{l} 
 +R^{\alpha}_{\beta\deltabar\gamma} f^{\beta}_{i}
 (f^{\gamma}_{\jbar} f^{\deltabar}_{\bar{k}} 
- f^{\gamma}_{\bar{k}} f^{\deltabar}_{\bar{j}}); \label{e:com1}\\
&f^\alpha_{\jbar|i|k}-f^\alpha_{\jbar|k|i}=\sqrt{-1}(h_{k\jbar}A_i^{\lbar}-h_{i\jbar} A_k^{\lbar})f^\alpha_{\lbar}
+ R^{\alpha}_{\beta\deltabar\gamma} f^\beta_{\jbar}
(f^\gamma_{i} f^{\deltabar}_{k} -f^\gamma_{k} f^{\deltabar}_{i}); \\
&f^{\alpha}_{i | j | \bar{k}} - f^{\alpha}_{i | \bar{k} | j} 
= \sqrt{-1}h_{j\bar{k}} f^{\alpha}_{i | 0} 
+ R_{i}{}^{l}{}_{j\bar{k}} f^{\alpha}_{l}
+ R^{\alpha}_{\beta\deltabar\gamma} f^{\beta}_{i}(f^{\gamma}_{j} f^{\bar{\delta}}_{\bar{k}} - f^{\gamma}_{\bar{k}} f^{\bar{\delta}}_{j}) \label{e:com2}
\end{align}
\end{lemma}
\begin{proof} The proof use the usual commutation relations for functions on CR manifolds as derived \cite{L2} and \cite{LiW}. For the first relation in \eqref{e:c1}, since $\Gamma^{\alpha}_{\beta\gamma} = \Gamma^{\alpha}_{\gamma\beta}$, we deduce from \eqref{e:covariantderivative} 
that
\begin{equation}
f^{\alpha}_{i | \jbar} - f^{\alpha}_{\jbar | i} = f^{\alpha}_{i \jbar} - f^{\alpha}_{\jbar i} = \sqrt{-1} f^{\alpha}_{0} h_{i \jbar},
\end{equation}
as desired. The proof of~\eqref{e:c2} is similar. To prove \eqref{e:com1}, we compute,
at a point $p$ under the assumption that the Christoffel symbols of $M$ at $p$ and those of $N$
at $f(p)$ vanish,
\begin{align}
f^{\alpha}_{i|\jbar | \bar{k}} 
& = (f^{\alpha}_{i|\jbar})_{,\bar{k}} \quad \text{(covariant derivative)}\notag \\
& = f^{\alpha}_{i\jbar\bar{k}} + (\Gamma^{\alpha}_{\beta\gamma} f^{\beta}_{i} f^{\gamma}_{\jbar} )_{,\bar{k}} \notag \\
& = f^{\alpha}_{i\jbar\bar{k}} 
+ \partial_{\delta}\Gamma^{\alpha}_{\beta\gamma} f^{\beta}_{i} f^{\gamma}_{\jbar} f^{\delta}_{\bar{k}} 
+ R^{\alpha}_{\beta\deltabar\gamma} f^{\bar{\delta}}_{\bar{k}} f^{\beta}_{i} f^{\gamma}_{\jbar}
\end{align}
Since the Christoffel symbols vanish at $f(p)$, one has $ \partial_{\delta}\Gamma^{\alpha}_{\beta\gamma} = \partial_{\gamma}\Gamma^{\alpha}_{\beta\delta}$ at $f(p)$, and thus
\begin{equation}
f^{\alpha}_{i|\jbar | \bar{k}} - f^{\alpha}_{i| \bar{k} | \jbar} 
= f^{\alpha}_{i\jbar\bar{k}} - f^{\alpha}_{i\bar{k}\jbar} 
+ R^{\alpha}_{\beta\deltabar\gamma} f^{\beta}_{i}
 (f^{\gamma}_{\jbar} f^{\deltabar}_{\bar{k}} 
- f^{\gamma}_{\bar{k}} f^{\deltabar}_{\bar{j}}),
\end{equation}
which, by the commutation relations for Tanaka-Webster covariant derivatives (see \cite{GL}), implies~\eqref{e:com1}. 

Similarly, with $\Gamma^\alpha_{\beta\gamma}(f(p))=0$, one has
\begin{eqnarray*}
f^\alpha_{\jbar| i|k}-f^\alpha_{\jbar|k|i}
&=&f^\alpha_{\jbar i k}-f^\alpha_{\jbar k i}+(\Gamma^\alpha_{\beta\gamma} f^\beta_{\jbar} f^{\gamma_i})_{,k}-(\Gamma^\alpha_{\beta\gamma} f^\beta_{\jbar} f^{\gamma}_k)_{,i}\\
&=&-\sqrt{-1}(h_{i\jbar} A^{\lbar}_k-h_{k\jbar} A^{\lbar}_i )f^\alpha_{\lbar} + R^{\alpha}_{\beta\deltabar\gamma} f^\beta_{\jbar}
(f^\gamma_{i} f^{\deltabar}_{k} -f^\gamma_{k} f^{\deltabar}_{i})\\
&=&\sqrt{-1}(h_{k\jbar}A_i^{\lbar}-h_{i\jbar} A_k^{\lbar})f^\alpha_{\lbar} + R^{\alpha}_{\beta\deltabar\gamma} f^\beta_{\jbar}
(f^\gamma_{i} f^{\deltabar}_{k} -f^\gamma_{k} f^{\deltabar}_{i}).
\end{eqnarray*}
This proves (3.8). The proof of \eqref{e:com2} is similar and is omitted.
\end{proof}

\begin{proof}[Proof of Theorem~\ref{thm:crsiu}]
We compute covariant derivative, using \eqref{e:com1}, (3.8) and \eqref{e:com2},
\begin{align}
 D_{k}(B_{i\jbar} f^{\alpha}) 
 & = f^{\alpha}_{i | \jbar | k} 
 -\frac{1}{m} f^{\alpha}_{l | \bar{p} | k} h^{l\bar{p}} \notag\\
 & = f^{\alpha}_{\jbar | i | k} + \sqrt{-1} f^{\alpha}_{0 | k} h_{i\jbar} 
-\frac{1}{m} f^{\alpha}_{l | \bar{p} | k} h^{l\bar{p}} h_{i\jbar} \notag \\
 & = f^{\alpha}_{\jbar | k | i} + \sqrt{-1}(h_{k\jbar} A_{i}\!^{\bar{l}} - h_{i\jbar} A_{k}\!^{\bar{l}}) f^{\alpha}_{\bar{l}} \notag \\
 & \quad + R^{\alpha}_{\rho\deltabar\gamma} f^{\rho}_{\jbar} (f^{\gamma}_{i} f^{\deltabar}_{k} - f^{\gamma}_{k} f^{\deltabar}_{i}) + \sqrt{-1} f^{\alpha}_{0 | k} h_{i\jbar}
 -\frac{1}{m} f^{\alpha}_{l | \bar{p} | k} h^{l\bar{p}}h_{i\jbar} \notag \\
& = f^{\alpha}_{\jbar | k | i} + \sqrt{-1}(h_{k\jbar} A_{i}\!^{\bar{l}} - h_{i\jbar} A_{k}\!^{\bar{l}}) f^{\alpha}_{\bar{l}} \notag \\
 & \quad+ R^{\alpha}_{\rho\deltabar\gamma} f^{\rho}_{\jbar} (f^{\gamma}_{i} f^{\deltabar}_{k} - f^{\gamma}_{k} f^{\deltabar}_{i})
-\frac{1}{m} f^{\alpha}_{\bar{p} | l | k} h^{l\bar{p}}h_{i\jbar}.
\end{align}
Taking the trace over $k$ and $j$ (using the Levi matrix $h^{k\jbar}$),
\begin{equation} \label{e:bp1}
 h^{k\bar{j}} D_{k}(B_{i\jbar} f^{\alpha}) 
= \frac{m-1}{m} P_i f^{\alpha} 
+ R^{\alpha}_{\rho\deltabar\gamma} f^{\rho}_{\jbar} (f^{\gamma}_{i} f^{\deltabar}_{k} - f^{\gamma}_{k} f^{\deltabar}_{i}) h^{k\jbar}
\end{equation}
From \eqref{e:bp1}, we see that if the target manifold $N$ is flat and $m\geq 2$,
then any CR-pluriharmonic map satisfies $P_i f^{\alpha} = 0$. This is Graham and Lee's result in \cite{GL}.

As in \cite{GL}, we consider the tensor $E$ on $M$ defined by
\begin{equation}
 E_{\jbar} = g_{\alpha\betabar} f^{\betabar}_{\bar{l}} (B_{i\jbar} f^{\alpha}) h^{i\bar{l}}.
\end{equation}
Then, the divergent $\delta E$ is
\begin{multline}
 E_{\jbar}{}^{,\jbar}
 = g_{\alpha\betabar} f^{\betabar}_{\bar{l} | k} (B_{i\jbar} f^{\alpha}) h^{i\bar{l}} h^{k\jbar}
 + g_{\alpha\betabar} f^{\betabar}_{\bar{l}} (B_{i\jbar} f^{\alpha})_{| k} h^{i\bar{l}} h^{k\jbar}\\
 = |B_{i \jbar} f^{\beta}|^2 + \frac{m-1}{m} \langle Pf, \bar{\partial}_b \bar{f}\rangle
 + R_{\rho\deltabar\gamma\betabar} f^{\betabar}_{\bar{l}} f^{\rho}_{\jbar}
(f^{\gamma}_{i} f^{\deltabar}_{k} - f^{\gamma}_{k} f^{\deltabar}_{i}) h^{i\bar{l}} h^{k\jbar}.
\end{multline}
Here, we used the fact that $B_{i\jbar} f^{\alpha}$ is trace-free.
Taking integral both sides over closed manifold $M$, we obtain
\begin{multline}
- \frac{m-1}{m}\int_M \langle Pf, \bar{\partial}_b \bar{f}\rangle
 = \int_M |B_{i\jbar} f^{\beta}|^2 
 + \int_M R_{\rho\deltabar\gamma\betabar} f^{\betabar}_{\bar{l}} f^{\rho}_{\jbar}
(f^{\gamma}_{i} f^{\deltabar}_{k} - f^{\gamma}_{k} f^{\deltabar}_{i}) h^{i\bar{l}} h^{k\jbar}.
\end{multline}
This proves \eqref{e:po}. To prove \eqref{e:po2}, we consider
\begin{equation}
 F_{\bar{l}} = g_{\alpha\betabar} f^{\alpha}_{\jbar | k} f^{\betabar}_{\bar{l}} h^{k\jbar}.
\end{equation}
The divergent $\delta F$ is 
\begin{align}
 F_{\lbar}\!^{,\lbar} 
 & = g_{\alpha\betabar} f^{\alpha}_{\jbar | k | i} f^{\betabar}_{\bar{l}} h^{k\jbar} h^{i\bar{l}} 
 + g_{\alpha\betabar} f^{\alpha}_{\jbar | k} f^{\betabar}_{\bar{l} | i} h^{k\jbar} h^{i\bar{l}}\\
 & = g_{\alpha\betabar} \left(P_i f^{\alpha} - \sqrt{-1}m A_{i}\!^{\bar{k}} f^{\alpha}_{\bar{k}} \right) f^{\betabar}_{\bar{l}} h^{k\jbar} h^{l\ibar} 
 + g_{\alpha\betabar} \tau^{\alpha}[f] \overline{ \tau ^{\betabar}[f]}.
\end{align}
Taking integration over $M$, we obtain the desired equality.
\end{proof}

\section{Rigidity theorem/Proof of Corollary 1.3}
In this section, we prove Corollary~\ref{cor:rigidity}. For the part (a), suppose $N$ has
strongly semi-negative curvature, then 
\begin{eqnarray}
\lefteqn{R_{\rho\deltabar\gamma\betabar} f^{\betabar}_{\bar{l}} f^{\rho}_{\jbar}
(f^{\gamma}_{i} f^{\deltabar}_{k} - f^{\gamma}_{k} f^{\deltabar}_{i}) h^{i\bar{l}} h^{k\jbar}} \notag\\
&=& {1\over 2} R_{\gamma\deltabar\rho\betabar}
(f^{\gamma}_{i} f^{\deltabar}_{k} - f^{\gamma}_{k} f^{\deltabar}_{i})
(\overline{f^{\beta}_{l}f^{\bar{\rho}}_{j} - f^{\beta}_{j}f^{\bar{\rho}}_{l}} ) 
h^{k\jbar}h^{i\bar{l}} \geq 0.
\end{eqnarray}
This and \eqref{e:po} imply that 
\begin{equation}
 \int_M \langle Pf , \bar{\partial}_b \bar{f} \rangle \leq 0.
\end{equation}
The equality holds if and only if both terms in the right of~\eqref{e:po} vanish; in particular,
\begin{equation}
 B_{i\jbar} f^{\alpha} = 0,
\end{equation}
and therefore, $f$ is CR-pluriharmonic.
\medskip

Part (b) follows from the following more general theorem.
\begin{theorem}\label{th:siusampson}
Let $(M^{2m+1}, \theta)$ be a closed pseudo-Hermitian CR manifold with $CR$ dimension $m$, $(N, g)$ a K\"{a}hler manifold,
and $f\colon M \to N$ a pseudo-Hermitian harmonic map. Suppose also that $M$ has vanishing Webster torsion, $N$ has strongly negative curvature, and $df|_H$ 
has (real) rank at least 4 on a dense set of $M$, then $f$ is a CR or anti CR map. Furthermore, if $f$ is an immersion, then $f$ must be CR. More generally, if the curvature of $N$
is negative of order $k$ and $df|_H$ has rank at least $2k$ on a dense set, then the same conclusion holds.
\end{theorem}
Here, as defined in \cite{Siu}, the curvature tensor $R_{\alpha\betabar \gamma \deltabar}$ of $N$ is said to be \emph{negative of order $k$} if it is strongly semi-negative and satisfies the following: If $A = (A^{\alpha}_{\ibar})$ and $B=(B^{\alpha}_{i})$ are any two $m\times k$ matrices with
\[
\rk \begin{pmatrix} A & B \\ \bar{B} & \bar{A} \end{pmatrix} = 2k,
\]
and if
\begin{equation*}
 \sum_{\alpha\beta\gamma\delta} R_{\alpha\betabar\gamma\deltabar} 
 \xi^{\alpha\betabar}_{\ibar\jbar}\overline{\xi^{\gamma\deltabar}_{\ibar\jbar}} =0, \quad \text{where} \quad \xi^{\alpha\betabar}_{\ibar\jbar} 
 = A^{\alpha}_{\ibar}\overline{B^{\beta}_{j}} - A^{\alpha}_{\jbar} \overline{B^{\beta}_i}
\end{equation*}
for all $1\leq i, j\leq k$, then either $A=0$ or $B=0$. As proved in \cite[Lemma~2]{Siu}, if the curvature of $N$ is strongly negative, then it is negative of order~2.

\begin{proof}[Proof of Theorem~\ref{th:siusampson}] The proof follows the lines in \cite{Siu} and is included here for completeness.
Let $U$ be an open connected open subset of $M$ such that the rank of $df|_H$ over $\mathbb{R}$
is at least $2k$ on $U$. Fix $p\in M$, we shall prove that either $\dbb f(p) = 0$, or $\partial_b f(p) = 0$; since $U$ is dense in $M$, it suffices to prove in the case $p\in U$.
Let $K$ be the kernel of $df|_H \colon H_p \to T_qM$. Then $\dim_{\mathbb{R}} K \leq 2n - 2k$. By Lemma~1 in \cite{Siu}, 
 there exists a basis $g_1, g_2, \dots , g_n$ of $H_p$ over $\mathbb{C}$ such that for a tube of $k$ indices
 $1\leq i_1< \dots < i_k \leq n$ the intersection of $K$ with the $\mathbb{C}$-vector subspace of $H_p$ spanned by 
 $g_{i_1}, \dots , g_{i_k}$ is zero. Choose a holomorphic frame $\{Z_j\colon 1\leq j\leq n\}$ in a neighborhood of $p$
 such that $g_j = Z_j + \bar{Z}_j$ and let $W$ be the complex vector subspace of $T^{1,0}_p \oplus T^{0,1}_p$, spanned by $\{Z_j|_p, \bar{Z}_j|_p \mid j =1, 2, \dots k\}$. 
 Then $df|_W$ has rank at least $2k$. Let $A = (f_{\ibar}^{\alpha})$ and $B=(f^{\alpha}_{i})$ at~$p$. Then
 \begin{equation}
 \rk \begin{pmatrix}
 A & B\\
 \bar{B} &\bar{A}
 \end{pmatrix}
 =2k
 \end{equation}
Since $M$ has vanishing torsion, i.e, $A^{ij} = 0$, and $\tau[f]=0$, from~\eqref{e:po2}, we deduce that $f$ is $\bar{\partial}_b$-pluriharmonic, and
\begin{equation}\label{e:r1}
 \sum_{\alpha\beta\gamma\delta} R_{\alpha\betabar\gamma\deltabar} 
 \xi^{\alpha\betabar}_{\ibar\jbar}\overline{\xi^{\gamma\deltabar}_{\ibar\jbar}} =0 \quad \text{at}\quad q',
\end{equation}
where
 \begin{equation}
 \xi^{\alpha\betabar}_{\ibar\jbar} = f^{\alpha}_{\ibar}(p) f^{\betabar}_{\jbar}(p) - f^{\alpha}_{\jbar}(p) f^{\betabar}_{\ibar}(p) 
 = A^{\alpha}_{\ibar}\overline{B^{\beta}_{j}} - A^{\alpha}_{\jbar} \overline{B^{\beta}_i}.
 \end{equation} 
Since $R_{\alpha\betabar\gamma\deltabar}$ is negative of order $k$, we deduce that for $1\leq i_1< \dots < i_k \leq n$,
either $f_{\jbar}^{\alpha} = 0$ for all $\alpha$ and $j=i_1, i_2,\dots, i_k$ or $f^{\alpha}_j = 0$ for all $\alpha$
and $j=i_1, i_2,\dots, i_k$. Because $k\geq 2$, it must follows that either $f_{\ibar}^{\alpha} = 0$ for all $\alpha$
and all $j$, or $f_{j}^{\alpha} = 0$ for all $\alpha$ and all~$j$.

To finish the proof, we need the following lemma which generalizes a previous result for sphere in $\mathbb{C}^{m+1}$ in \cite[Theorem~3.1]{LN}.
\begin{lemma}
Let $M$ be a strictly pseudoconvex CR manifold of dimension at least $5$. Suppose that $M$ admits a (local) transversal infinitesimal CR automorphism $X$ at every point. Assume that $g$ is a twice differentiable function on $M$ such that for any $p\in M$, either $\partial_b g(p)=0$, or $\bar{\partial}_b g(p) = 0$. Then $g$ is either CR, or anti CR.
\end{lemma}
\begin{proof}
We shall follow the idea in \cite[Theorem~3.1]{LN}. Let $A_0$ and $B_0$ be the closures of the interiors of the sets $\{\partial_b g=0\}$ and $\{\bar{\partial}_b g=0\}$, respectively. It follows that $A_0\cup B_0 =M$. If either $A_0$ or $B_0$ is empty, then the conclusion of the lemma is clear. Thus suppose that both sets are nonempty. We shall show that $g$ is a constant. By connectedness, it suffices to show that $g$ is locally constant. Therefore, we can assume that there is a frame $\{Z_{i}, Z_{\ibar},T\}$ on~$M$ such that $T=X$ is the given infinitesimal CR automorphism, $\{Z_i\}$ is a holomorphic frame, and $Z_{\ibar} =\overline{Z_{i}}$. We have the following identities
\begin{align}
	[Z_{\jbar} , Z_i] & = \sqrt{-1} h_{i\jbar} T + \Gamma^l_{\jbar i} Z_l - \Gamma^{\bar{l}}_{i\jbar} Z_{\lbar}, \label{e:c11}\\
	[Z_j, Z_i] &=\Gamma_{ji}^k Z_k - \Gamma_{ij}^k Z_k, \label{e:c2} \\
	[Z_j, T] &= -\Gamma_{0j}^k Z_k;\label{e:c3} 
\end{align}
where the $h_{i\jbar}$ and the Christoffel symbols are the Levi-form and the symbols corresponding to the pseudohermitian structure $\theta$ for which $T$ is the associated Reeb vector field \cite[p. 418]{L}. Notice that the vanishing of the components of the Webster torsion follows from the assumption that $T$ is an infinitesimal CR automorphism.

Let $K = A_0\cap B_0$ and $p\in K$. Then for each $j$, $Z_j g$ and $Z_{\jbar} g$ both vanish at $p$. Moreover, by continuity, $Z_kZ_{\jbar} g =0$ on $B_0$ and $Z_{\jbar} Z_k g =0$ on $A_0$. Therefore, from~\eqref{e:c11} we find that $Tg =0$ on $K$.

Let
\begin{equation}
G(p) = \begin{cases}
 0, \quad &\text{if}\ p\in B_0,\\
 Tg(p), \quad & \text{if}\ p\in A_0.
\end{cases}
\end{equation}
Then $G$ is continuous on $M$. We claim that $G$ is anti CR on $M$. In fact, on the interior of $A_0$, one computes, using~\eqref{e:c3},
\[
Z_j(Tg) = T(Z_jg) - \Gamma_{0j}^k Z_k g = 0.
\]
Then by continuity, $Z_j(G)$ vanishes on $A_0$ and so on the whole $M$; the claim immediately follows.

By well-known unique continuation for anti CR functions (see, e.g., \cite{LSW} for a detail proof, notice that $M$ is locally embeddable into $\mathbb{C}^{m+1}$), it follows that $G\equiv 0$ on~$M$. Whence, $Tg$ vanishes on $A_0$. Using similar argument for $\bar{g}$, we find that $Tg = \overline{T\bar{g}}$ also vanishes on $B_0$. Therefore,
\begin{equation}
\label{e:tg}
Tg = 0 \quad \text{on} \ M.
\end{equation}
From~\eqref{e:c11} and \eqref{e:tg}, we find that, for all $i, j$,
\begin{equation}\label{e:last}
	g_{i,\bar{j}} = g_{\bar{j}, i}.
\end{equation}
Since the left hand side vanishes on $A_0$, while the right hand side vanishes on $B_0$, both must vanish on $M$. Hence $g$ must be a constant.
\end{proof}
Applying the lemma for each component $f^{\alpha}$ (noticing that the vanishing of Webster torsion implies that the Reeb vector field $T$ is a transversal CR automorphism), we conclude that $f$ is either CR or anti-CR, as desired.

Finally, if $f$ is anti CR, then $f^{\alpha}_{j} = 0$. Therefore,
\begin{equation}
f^{\alpha}_ 0 = \frac{\sqrt{-1}}{m} (f^{\alpha}_{\ibar | j} - f^{\alpha}_{j | \ibar}) h^{j\ibar} = 0.
\end{equation}
This equation cannot happen when $f$ is an immersion. 
\end{proof}

\end{document}